\documentclass[12pt]{article}

\usepackage{amsmath,amssymb,amsthm,amsfonts,graphics,graphicx,listings,float,titling,bm,geometry,ulem}
\usepackage[hidelinks]{hyperref}
\usepackage{enumerate}
\usepackage{stmaryrd}
\theoremstyle{plain}
\newtheorem{theorem}{Theorem}
\newtheorem{lemma}[theorem]{Lemma}

\newtheorem{proposition}[theorem]{Proposition}
\newtheorem{claim}[theorem]{Claim}

\newcommand{\floor}[1]{\left\lfloor #1 \right\rfloor}
\newcommand{\ceil}[1]{\left\lceil #1 \right\rceil}
\newcommand{\eps}{\varepsilon}

\newcommand{\ex}{\text{ex}}

\usepackage{thm-restate}

\usepackage{fontawesome}
\newcommand{\flags}{\faCogs}

\newcommand{\vc}[1]{\ensuremath{\vcenter{\hbox{#1}}}}

\usepackage{tikz}
\tikzset{flag_pic/.style={scale=1}}  
\tikzset{unlabeled_vertex/.style={inner sep=1.7pt, outer sep=0pt, circle, fill}} 
\tikzset{labeled_vertex/.style={inner sep=2.2pt, outer sep=0pt, rectangle, fill=gray, draw=black}} 
\tikzset{edge_color0/.style={color=black,line width=1.2pt,opacity=0.5}} 
\tikzset{edge_color1/.style={color=red,  line width=1.2pt,opacity=1}} 
\tikzset{edge_color2/.style={color=blue, line width=1.2pt,opacity=1}} 
\tikzset{edge_color3/.style={color=green!80!black,line width=1.2pt}} 
\tikzset{edge_color4/.style={color=orange,  line width=1.2pt}} 
\tikzset{edge_color5/.style={color=blue, line width=1.2pt,dotted}} 
\tikzset{edge_color6/.style={color=green, line width=1.2pt,dotted}} 
\tikzset{edge_color7/.style={color=orange, line width=1.2pt}} 
\tikzset{edge_color8/.style={color=gray, line width=1.2pt}} 
\tikzset{edge_colorroot/.style={color=red, line width=1.7pt}} 
\tikzset{edge_thin/.style={color=black}} 
\tikzset{edge_hidden/.style={color=black,dotted,opacity=0}} 
\tikzset{vertex_color1/.style={inner sep=1.7pt, outer sep=0pt, draw, circle, fill=red}} 
\tikzset{vertex_color2/.style={inner sep=1.7pt, outer sep=0pt, draw, circle, fill=blue}} 
\tikzset{vertex_color3/.style={inner sep=1.7pt, outer sep=0pt, draw, circle, fill=green}} 
\tikzset{vertex_color4/.style={inner sep=1.7pt, outer sep=0pt, draw, circle, fill=pink}} 
\tikzset{labeled_vertex_color1/.style={inner sep=2.2pt, outer sep=0pt, draw, rectangle, fill=red}} 
\tikzset{labeled_vertex_color2/.style={inner sep=2.2pt, outer sep=0pt, draw, rectangle, fill=blue}} 
\tikzset{labeled_vertex_color3/.style={inner sep=2.2pt, outer sep=0pt, draw, rectangle, fill=green!70!black}} 
\tikzset{labeled_vertex_color4/.style={inner sep=2.2pt, outer sep=0pt, draw, rectangle, fill=pink}} 

\def\outercycle#1#2{ 
\pgfmathtruncatemacro{\plusone}{#1+1} 
\pgfmathtruncatemacro{\minusone}{#1-1} 
\draw  \foreach \x in {0,...,\minusone}{(0.5*\x,0) coordinate(x\x)};}

\def\drawhypervertex#1#2{ \pgfmathtruncatemacro{\plusone}{#1+1}  \draw[edge_color2] (x#1)++(0,-0.2-0.2*#2)+(-0.2,0) -- +(0.2,0) +(0,0) node[fill=white,outer sep=0,inner sep=0]{\tiny \plusone};} 

\def\drawhypervertexcolor#1#2#3{ \pgfmathtruncatemacro{\plusone}{#1+1} \draw[edge_color#3] (x#1)++(0,-0.2-0.2*#2)+(-0.2,0) -- +(0.2,0)  +(0,0) node[fill=white,outer sep=0,inner sep=0]{\tiny \plusone};} 

\def\drawhyperedge#1#2{ \draw[dotted] (x0)++(0,-0.2-0.2*#1)--++(0.5*#2-0.5,0);
\path (0,-0.4-0.2*#1) -- (0,0); 
}

\def\labelvertex#1{\pgfmathtruncatemacro{\vertexlabel}{#1+1 } \draw (x#1) node{\color{yellow}\tiny\vertexlabel}; }

\tikzset{
vtx/.style={inner sep=1.1pt, outer sep=0pt, circle, fill,draw}, 
vtxl/.style={inner sep=1.1pt, outer sep=0pt, rectangle, fill=yellow,draw=black}, 
hyperedgeP/.style={fill=pink,opacity=0.5,draw=black}, 
hyperedgeB/.style={fill=blue,opacity=0.5,draw=black}, 
}

\newcommand{\CM}[1]{\textcolor{black}{#1}}

\usepackage{mathtools}

\title{The Hypergraph {T}ur\'{a}n Densities of 
Tight Cycles Minus an Edge}
\author{Bernard Lidick\'{y}\thanks{Department of Mathematics, Iowa State University, Ames, IA. E-mail: \texttt{lidicky@iastate.edu}. Research of this author is supported in part by NSF grant DMS-2152490 and Scott Hanna professorship.}
\and
Connor Mattes\thanks{Sandia National Laboratories. Livermore, California, USA. E-mail: \texttt{clmatte@sandia.gov}. Sandia National Laboratories is a multimission laboratory managed and operated by National Technology \& Engineering Solutions of Sandia, LLC, a wholly owned
subsidiary of Honeywell International Inc., for the U.S. Department of Energy's National Nuclear Security Administration under contract DE-NA0003525. Research on this project was partially supported by NSF grant DMS-2152498.}
\and
Florian Pfender\thanks{Department of Mathematical and Statistical Sciences, University of Colorado Denver. E-mail: \texttt{Florian.Pfender@ucdenver.edu}. Research is partially supported by NSF grant DMS-2152498.}
}

\newcommand{\x}[1]{\vc{\tikz\draw[black] (0,0) circle (0.8ex) node{\scriptsize$#1$};}}

\newcommand{\exx}[2]{\vc{\tikz\draw[black,fill=white] (0.8ex,0)--(0.5cm-0.8ex,0) (0,0) circle (0.8ex) node{\scriptsize$#1$}  (0.5,0) circle (0.8ex) node{\scriptsize$#2$};}}
\newcommand{\e}[1]{{\exx #1}}

\newcommand{\pxx}[3]{\vc{\tikz\draw[black,fill=white] 
(0.8ex,0)--(0.5cm-0.8ex,0) 
(0.5cm+0.8ex,0)--(1cm-0.8ex,0) 
(0,0) circle (0.8ex) node{\scriptsize$#1$}
(0.5,0) circle (0.8ex) node{\scriptsize$#2$}
(1.0,0) circle (0.8ex) node{\scriptsize$#3$}
;}}
\newcommand{\p}[1]{{\pxx #1}}

\newcommand{\kxx}[3]{\vc{\tikz\draw[black,fill=white] 
(0.8ex,0)--(0.5cm-0.8ex,0) 
(0.5cm+0.8ex,0)--(1cm-0.8ex,0) 
(0,0.8ex) to[bend left](1cm, 0.8ex) 

(0,0) circle (0.8ex) node{\scriptsize$#1$}
(0.5,0) circle (0.8ex) node{\scriptsize$#2$}
(1.0,0) circle (0.8ex) node{\scriptsize$#3$}
;}}
\newcommand{\kkk}[1]{{\kxx #1}}

\begin{document}

\maketitle











\begin{abstract}
    A tight $\ell$-cycle minus an edge $C_\ell^-$ is the $3$-graph on the vertex set $[\ell]$, where any three consecutive vertices in the string $123\ldots\ell 1$ form an edge. We show that for every $\ell\ge 5$, $\ell$ not divisible by $3$, the extremal number is
    \[
    \ex\left(C_\ell^-,n\right)=\tfrac1{24}n^3+O(n\ln n)=\left(\tfrac14+o(1)\right){n\choose 3}.
    \]
    We 
    determine the extremal graph up to $O(n)$ edge edits.
\end{abstract}

\section{Introduction}\label{intro}

Maybe the oldest and most fundamental question in extremal hypergraph theory is to maximize the number of edges in an $r$-uniform hypergraph (aka $r$-graph) on $n$ vertices, which does not contain a given hypergraph $H$ as a subgraph. The resulting function is called the {\em extremal number} of $H$, and denoted by $\ex(H,n)$.

We can normalize this quantity, and an easy averaging argument shows that the limit
\[
\pi(H)=\lim_{n\to\infty}\frac{\ex(H,n)}{{n\choose r}}
\]
exists for every $r$-graph $H$. This limit is called the {\em Tur\'an density} of $H$. For $2$-graphs, a celebrated theorem by Erd\H{o}s, Stone, and Simonovits  relates the Tur\'an density to the chromatic number $\chi(H)$.
\begin{theorem}[\cite{ErdosSimonovits, ErdosStone}]
    For every $2$-graph $H$, $\pi(H)=\frac{\chi(H)-2}{\chi(H)-1}$.
\end{theorem}
For $r\ge 3$, this question is wide open. Erd\H{o}s \cite{Erdos1964} showed  that $\pi(H)=0$ if and only if $H$ is an $r$-partite $r$-graph. Other than this, the answer is known for only very few $r$-graphs or classes of $r$-graphs. 

Restricting ourselves to $3$-graphs from now on, finding the Tur\'an density for $K_4=K_4^3$, the complete $3$-graph on $4$ vertices, may be the most famous open problem in extremal hypergraph theory. Even if we delete one edge from $K_4$ to create $K_4^-$, the smallest not $3$-partite $3$-graph, $\pi(K_4^-)$ is unknown. The conjectured values for these two densities are $\pi(K_4)=\frac59$ \cite{Turan1941} and $\pi(K_4^-)=\frac27$ \cite{Mubayi2003}. 

Three small $3$-graphs we know the Tur\'an density for are the Fano plane $F$ on $7$ vertices, and the {\em books} $B_{3,2}$ and $B_{3,3}$ on $5$ vertices and edge sets $\{123,124,345\}$ and $\{123,124,125,345\}$, respectively. We have $\pi(F)=\frac34$ \cite{Furedi2000}, $\pi(B_{3,2})=\frac29$ \cite{Frankl1983}, and $\pi(B_{3,3})=\frac49$ \cite{Furedi2003}. For each of these three $3$-graphs, the structure of the extremal construction is rather simple. The vertices are partitioned into two or three parts of appropriate sizes, and edges are completely determined by the parts the vertices belong to. 

For known results and fundamental techniques, a survey by Keevash \cite{Keevash2011} is a good resource, for open problems see also a collection by Mubayi, Pikhurko and Sudakov \cite{mubayi2011}. 
For recent updates see Balogh, Clemen and Lidick\'y~\cite{FAsurvey}.

Since these surveys, two recent results are exciting developments.
For this, let us consider the tight $\ell$-cycle $C_\ell$, a $3$-graph on vertex set $[\ell]$, and the edges encoded in the string $123\ldots\ell12$, where the edges are exactly all $3$-sets which appear consecutively in the string. Similarly, we can define the tight $\ell$-cycle minus an edge $C_\ell^-$ by the string  $123\ldots\ell1$. Observe that $K_4=C_4$ and $K_4^-=C_4^-$, so these two classes contain the two most notorious $3$-graphs for which the Tur\'an density is unknown. If $\ell$ is a multiple of $3$, then $C_\ell$ and $C_\ell^-$ are $3$-partite, so $\pi(C_\ell)=\pi(C_\ell^-)=0$. For other large enough $\ell$, we have the following two Theorems by Kam\v{c}ev, Letzter, and Pokrovskiy \cite{kamcev2023turan}, and by Balogh and Luo \cite{balogh2023turan}.

\begin{theorem}[\cite{kamcev2023turan}]
     There is a constant $L$ such that $\pi\left(C_\ell\right)=2\sqrt{3}-3$, for every $\ell\ge L$ not divisible by $3$.
\end{theorem}
\begin{theorem}[\cite{balogh2023turan}]\label{thmjozsi}
     There is a constant $L$ such that $\pi\left(C_\ell^-\right)=\frac14$, for every $\ell\ge L$ not divisible by $3$.
\end{theorem}

Both the previous statements are conjectured to be true already for $L=5$ (they are false for $\ell=4$), but the proofs only work for very large $L$.
In this paper, we add to the very meager set of known results for small $3$-graphs, and first determine the Tur\'an density of $C_5^-$. 
The limit object is the same as in Theorem~\ref{thmjozsi}, and  has a much more intricate structure, described later, than the limit objects for the results on other small $3$-graphs mentioned above.

\begin{theorem}\label{Turan}
    \[
    \pi\left(C_5^-\right)=\frac{1}{4}.
    \]
\end{theorem}
Using blow-up arguments from~\cite{balogh2023turan}, this implies the Tur\'an density for  $C_\ell^-$ for all $\ell\ge 5$, completing the statement in Theorem~\ref{thmjozsi}. 
\begin{theorem}\label{Turangen}
For every $\ell\ge 5$,
    \[
    \pi\left(C_{\ell}^-\right)=\begin{cases}
        0,&\text{ if }3\mid \ell,\\
        \frac{1}{4}, &\text{ if }3\nmid \ell.
    \end{cases}
    \]
\end{theorem}
After proving the main lemma of this paper in Section~\ref{main2}, we prove Theorems~\ref{Turan} and \ref{Turangen} in Section~\ref{limit}.

\begin{figure}[h]
\begin{center}  
\begin{tikzpicture}


\draw[hyperedgeP] 
(288:1.5) to[out=80,in=300,looseness=1.2] (72:1.5) to[out=310,in=140,looseness=1.2] (0:1.5) to[out=210,in=70,looseness=1.2] (288:1.5)
;

\draw[hyperedgeP] 
(0:1.5) to[out=100,in=0,looseness=1.2] (72:1.5) to[out=150,in=30,looseness=1.2] (144:1.5) to[out=50,in=90,looseness=1.8] (0:1.5);

\draw[hyperedgeB] 
(72:1.5) to[out=200,in=10,looseness=1.2] (144:1.5) to[out=280,in=80,looseness=1.2] (216:1.5) to[out=70,in=210,looseness=1.2] (72:1.5);

\draw[hyperedgeB] 
(144:1.5) to[out=250,in=140,looseness=1.2] (216:1.5) to[out=280,in=180,looseness=1.2] (288:1.5) to[out=200,in=230,looseness=1.8] (144:1.5);

\draw 
(0:1.5) node[vtx,label=below:2,edge_color1]{.}
(216:1.5) node[vtx,label=right:5,edge_color2]{.}
(72:1.5) node[vtx,label=below:3]{.}
(144:1.5) node[vtx,label=left:4]{.}
(288:1.5) node[vtx,label=right:1]{.}
;
\end{tikzpicture}
\hskip 3em
\begin{tikzpicture}[flag_pic]\outercycle{6}{0}
\draw (x0) node[unlabeled_vertex]{.};\draw (x1) node[unlabeled_vertex,edge_color1]{.};\draw (x2) node[unlabeled_vertex]{.};\draw (x3) node[unlabeled_vertex]{.};\draw (x4) node[unlabeled_vertex,edge_color2]{.};
\labelvertex{0}\labelvertex{1}\labelvertex{2}\labelvertex{3}\labelvertex{4}
\drawhypervertexcolor{0}{0}{1}
\drawhypervertexcolor{1}{0}{1}
\drawhypervertexcolor{2}{0}{1}
\drawhyperedge{1}{5}
\drawhypervertexcolor{1}{1}{1}
\drawhypervertexcolor{2}{1}{1}
\drawhypervertexcolor{3}{1}{1}
\drawhyperedge{2}{5}
\drawhypervertexcolor{2}{2}{2}
\drawhypervertexcolor{3}{2}{2}
\drawhypervertexcolor{4}{2}{2}
\drawhyperedge{3}{5}
\drawhypervertex{3}{3}{2}
\drawhypervertex{4}{3}{2}
\drawhypervertex{0}{3}{2}
\draw (1,-2) node{};
\end{tikzpicture} 
\hskip 3em
\begin{tikzpicture}
\draw 
(0,0) node[vtx,label=left:2,edge_color1]{.}
(1,0) node[vtx,label=right:5,edge_color2]{.}
(0,1) node[vtx,label=left:3](3){.}
(1,1) node[vtx,label=right:4](4){.}
(0.5,2) node[vtx,label=above:1](1){.}
;
\draw[edge_color1](3)--(1);
\draw[edge_color2](4)--(1);
\draw[edge_color1](3) to[bend left](4);
\draw[edge_color2](3)to[bend right](4);
\draw (1,-1) node{};
\end{tikzpicture}
\end{center}
\caption{$C_5^-$ drawn with hyperedges, as vertex-edge incidence, and as links of the two independent vertices}
\end{figure}

Furthermore, we show a stronger result, as we asymptotically determine the extremal $3$-graphs for all $C_\ell^-$ with $\ell\ge 5$ and not divisible by $3$.
Let the $3$-graph $H_n$ be the iterated balanced blow-up of an edge on $n$ vertices. In other words, for 
\[
i=\floor{\tfrac{n}3},~j=\floor{\tfrac{n+1}3},~k=\floor{\tfrac{n+2}3},
\]
inductively start with the three $3$-graphs $H_i$, $H_j$, and $H_k$, and add all edges spanning all three graphs. The start of the induction are the edge free graphs $H_1$ and $H_2$. Note that $H_n$ has edge density $\frac14+o(1)$. For any $\ell\ge 4$, $C_\ell^-$ is tightly connected, meaning that any two edges are connected by a tight walk, a sequence of at least three vertices where every three consecutive vertices span an edge. Thus, any copy of $C_\ell^-$ in $H_n$ would have to be contained in the same tight block of $H_n$. The tight blocks of $H_n$ are $3$-partite, which shows that $H_n$ contains no copy of $C_\ell^-$ unless $\ell$ is divisible by $3$.

Let $F_{\ell,n}$ be an extremal graph for $C_\ell^-$, a $C_\ell^-$-free $3$-graph on $n$ vertices with a maximum number of edges. 
We will show that, for sufficiently large $n$, $F_{\ell,n}$ is a balanced blow-up of an edge.

\begin{restatable}{theorem}{restatethmmain}\label{thmmain}
For $\ell\ge 5$ not divisible by $3$, let $F_{\ell,n}$ be a $C_\ell^-$-free $3$-graph on $n$ vertices with a maximum number of edges.
    For $n$ sufficiently large, there is a partition $V(F_{\ell,n})=X_1\cup X_2\cup X_3$ such that $v_1v_2v_3\in E(F_{\ell,n})$ for all $v_1\in X_1,v_2\in X_2, v_3\in X_3$, and $|X_1|,|X_2|,|X_3|\in\{\floor{\tfrac{n}3},\ceil{\tfrac{n}3}\}$.
\end{restatable}


Note that Theorem~\ref{thmmain} implies that there are no edges in $F_{\ell,n}$ of the type $u_iv_iv_j$ with $u_i,v_i\in X_i$, $v_j\in X_j$, $i\ne j$, as this 
together with appropriately chosen $\ell-3$ other vertices would induce a graph containing a copy of a $C_\ell^-$.

This implies in a strong sense that $F_{\ell,n}$ and $H_n$ converge to the same limit object, a hypergraphon. We discuss a bit of this perspective in Section~\ref{limit}, before we prove Theorem~\ref{thmmain} in Section~\ref{main}. 

We can use the theorem inductively to show that  for some large enough $M$, for every $k$ and $3^{k+1}M>n\ge 3^kM$, $F_{\ell,n}$ agrees on the first $k$ levels of $H_n$ from the outside in. This shows that $F_{\ell,n}$ and $H_n$ are isomorphic up to changing $3^k{3M\choose 3}=O(n)$ edges. In particular, since we can determine the number of edges of $H_n$ up to a small error, 
we have

\begin{restatable}{theorem}{restateexno}\label{exno}
    \[
\left|\ex\left(C_\ell^-,n\right)-\frac{n^3}{24}\right|<\frac16 n\log_3 n+O(n).
\]
\end{restatable}
In general, we can not expect that $F_{\ell,n}=H_n$, since modifying $H_s$ for small $s$ yields a few extra edges without creating a $C_\ell^-$. 
For example for $\ell=5$, $H_4$ is the $4$-vertex graph with $2$ edges, while $F_{5,4}$ is complete. Replacing every vertex of a single edge with a complete graph on $4$ vertices creates a $C_5^-$-free graph on $12$ vertices with $64+12=76$ edges, while $H_{12}$ contains only $64+6=70$ edges. Taking this further, we can construct a graph on $n=4\cdot 3^k$ vertices with $2\cdot 3^k=\frac{n}2$ extra edges compared to $H_n$. 
For $5\le s\le 8$, $F_{5,s}$ is the graph containing exactly all ${s-1\choose 2}$ edges containing a single vertex, verified by computer by an exhaustive search \flags. For $n\ge 9$, the best construction we know is as follows. Use the same recursion as in the construction of $H_n$, with the one difference that we use the known extremal graphs $F_{5,3}, F_{5,4},\ldots ,F_{5,8}$ instead of recursing all the way down to $H_1$ and $H_2$.
We wonder if this is in fact the unique extremal construction for all $n$.

For $\ell\ge 7$, we observe that the extremal graph $F_{\ell,s}$ is complete for every $s<\ell$, but we have not explored cases with $s\ge \ell$. The resulting construction on $(\ell-1)3^k$ vertices yields an extra $(\tfrac{\ell^2}{8}+O(\ell))n$ edges compared to $H_{(\ell-1)3^k}$, with the $O(\ell)$ in terms of growing $\ell$.
%




\section{Flag Algebra Methods}

Our method relies on the theory of flag algebras  developed by Razborov~\cite{razborov2007}.
Flag algebras can be used as a general tool to attack problems from extremal combinatorics. In this paper, we use it \CM{to bound densities of} $3$-graphs, $2$-graphs, and even $0$-graphs (i.e. graphs with no edges), often with colors added to vertices and/or edges.

\CM{The {\em plain flag algebra method}  by Razborov computationally automates many techniques from the theory of flag algebras.}
For a more thorough explanation of the method avoiding the technicalities of~\cite{razborov2007}, one may look at Section 4 of~\cite{RBtri}.
A typical application of this method provides asymptotic bounds on densities of substructures \CM{subject to constraints on the densities of other substructures}.
\CM{For example, in Lemma \ref{edgeBound}, we give an upper bound on the maximum number of edges in a $C_5^-$-free graph.}
To get \CM{accurate} bounds, true inequalities and equalities involving the densities of substructures are combined using semidefinite programming. 
\CM{Often, one can create more accurate bounds 
by considering equalities and inequalities on larger substructures. However, the size of these semi-definite programs}
quickly reach the bounds of the capabilities of even the largest computers.
Certificates for the truth of the statements \CM{proven by the plain flag algebra method} appear in the form of sums of squares. Due to their sheer size required in this paper, they are impractical to print, and checking them involves a computer. We provide our programs and certificates at
\url{http://lidicky.name/pub/c5-/}.
Lemmas and Claims using flag algebras and other computer assistance are indicated by \flags.

In some applications, the bounds from flag algebra are asymptotically sharp.
Obtaining an exact \CM{description of the extremal structure} from sharp bounds usually consists of first
bounding the densities of some small substructures by $o(1)$.
Often, these substructures can be read off from the flag algebra computation. 
Using a removal lemma applicable to the structure, one can get rid of all these substructures without changing the asymptotic densities of the extremal object.
From this, one can extract a lot of information about the structure. 
Finally, stability arguments can sometimes be used to extract the precise extremal object. While the other known exact results on Tur\'an densities of three small $3$-graphs mentioned in Section~\ref{intro} were not proven using flag algebras, the results for the two books can be re-proved with this near automatic approach. The corresponding computation for the Fano plane is too large, even with the very large computers we have access to.

For the result in this paper, bounds we get from the plain flag algebra method are not sharp. In our experience, this is typical when the extremal construction is more complicated than can be easily captured by the density of small substructures, as it is the case for most iterated constructions. We have encountered similar issues before on $2$-graphs, see \cite{C5, RBtri, c5frac}. In these previous applications 
\CM{ of the plain flag algebra method,  we were able to determine the extremal graphs using stability arguments paired with bounds on the densities of a few subgraphs} appearing with high density in the top level of the iterated extremal construction.

That same approach is not sufficient in this application. Looking merely at the top level does not give us \CM{ good enough bounds}. 
We instead also include densities of  \CM{subgraphs in } the next level.
We then extensively explore the local structure of the link of a vertex, the $2$-graph induced by the edges incident to one vertex, with the flag algebra method. The only time we stretched the limits of our available computation power is for Lemma~\ref{flag}, which took more than a month on a large cluster. All other computations are moderate in size and could be run on a personal computer in reasonable time.

\section{The Main Lemma}\label{main2}

For any two hypergraphs $H$ on $k$ vertices and $G$ on $n>k$ vertices, let 
\[
p(H,G)=\frac{|\{X\subset V(G):G[X]\simeq H\}|}{{n\choose k}}
\]
denote the induced density of $H$ in $G$. With a slight abuse of notation to improve readability, we will often just write $H$ instead of $p(H,G)$ in equations of subgraph densities where the large graph $G$ is clear from context. A first instance of this notation is the following proposition.

\begin{proposition}\label{noK4-}
    In any $C_5^-$-free $3$-graph $G$ on $n$ vertices, $K_4^-+K_4\le\frac4{n-3}$.
\end{proposition}

\begin{proof}
    Note that if we duplicate a vertex of degree $3$ in a $K_4^-$ or $K_4$, we create a graph containing $C_5^-$. 
    In other words, any set of three vertices can have at most one vertex forming an edge with all three vertex pairs.
    Thus, we can have at most ${n\choose 3}=\frac{4}{n-3}{n\choose 4}$ copies of $K_4^-$ and $K_4$ in $G$.
\end{proof}


We can eliminate all copies of $K_4^-$ via a removal lemma (see Section~\ref{limit}) without asymptotically changing the extremal graph. This new simplified graph is much easier to explore with the flag algebra method, as there are drastically\footnote{On 8 vertices there are 161,023 $\{C_5^-,K_4^-\}$-free  and 1,528,500 $C_5^-$-free 3-graphs. Flag algebras on a supercomputer can run on up to $\approx 200,000$ graphs.} fewer $3$-graphs which are $\{C_5^-,K_4^-\}$-free compared to all $C_5^-$-free $3$-graphs.
For this reason, we will study $\{C_5^-,K_4^-\}$-free graphs for most of this paper before coming back to $C_5^-$-free graphs. Similarly to $F_{5,n}$, define $G_n$ to be extremal $\{C_5^-,K_4^-\}$-free graphs on $n$ vertices.


The following main lemma falls slightly short of Theorem~\ref{thmmain}, but the proof of it contains most of the work.
\begin{lemma}[Main Lemma]\label{thmmain2}
    For $n$ sufficiently large, there is a partition $V(G_n)=X_1\cup X_2\cup X_3$ such that $v_1v_2v_3\in E(G_n)$ for all $v_1\in X_1,v_2\in X_2, v_3\in X_3$, and $|X_1|,|X_2|,|X_3|\in\{\floor{\tfrac{n}3},\ceil{\tfrac{n}3}\}$.
\end{lemma}
In this case, we are not aware of any value of $n$ where $H_n$ and $G_n$ differ, and it would not surprise us if $G_n=H_n$ for all $n$. In fact, Lemma~\ref{thmmain2} together with a blow-up argument can be used to show that $G_{3^k}=H_{3^k}$.

\begin{proof}[Proof of Lemma~\ref{thmmain2}]
The proof of this lemma is rather technical, and contains a number of lemmas and claims with their own separate proofs.

We apply the plain flag algebra method to get a first upper bound. It turns out that we never have to use this bound, but we provide it here to show what a simple application of the method achieves.

\begin{lemma}[\flags]\label{edgeBound}
$G_n$ has edge density less than 
$0.2502175 + o(1)$.
\end{lemma}

\begin{proof} 
A straight forward application of the plain flag algebra method using a large computation gives 
an asymptotic upper bound of
\[
\frac{14012179180700882058867945296422074007373856}{56000000000000000000000000000000000000000000} \approx 0.2502174853696586.
\]
\end{proof}
The graph $H_n$ has edge density $0.25+o(1)$. The bound in Lemma~\ref{edgeBound} is a bit larger than that, and it will sometimes be useful to work with a graph with edge density $0.25+o(1)$. For this, we randomly construct such a graph $\bar{G}_n$ from $G_n$ by uniformly at random deleting an appropriate number of  edges. 

Note that $\bar{G}_n$ is $\{K_4^-, C_5^-\}$-free,
since the class of $\{K_4^-, C_5^-\}$-free hypergraphs is closed under edge deletion, vertex deletion and vertex duplication. A standard symmetrization argument implies that $G_n$  is asymptotically regular.
Due to standard concentration results, $\bar{G}_n$ is also asymptotically regular with high probability.

\begin{claim}\label{deg}
    $\Delta(G_n)-\delta(G_n)<n$, and with probability $1-o(1)$, $\Delta(\bar{G}_n)-\delta(\bar{G}_n)=o(n\log n)$.
\end{claim}

Let $X_1,X_2,X_3\subset V(G_n)$ be disjoint sets,
and let $T=V(G)\setminus (X_1\cup X_2\cup X_3)$.
Call a triple of vertices funky if it is an edge and contains two vertices in $X_i$ and one vertex in $X_j$ for $i \ne j$, or if it is a non-edge and contains a vertex from each of $X_1,X_2,X_3$. Let $F$ denote the set of funky triples, with $F_1$ denoting the funky edges and $F_2$ denoting the funky non-edges. 

If we can show that $F = T = \emptyset$, then we know that $G_n$ is a blow-up of an edge. We now use a flag algebra argument to show that both $F$ and $T$ have relatively small cardinality on our way to show that $G_n$ must indeed be the blow-up of an edge. 
Let ${n\choose 3} f = |F|$, ${n\choose 3} f_i = |F_i|$, $tn=|T|$, and $x_i n = |X_i|$. We will assume without loss of generality that $x_1 \ge x_2 \ge x_3$.

Recall that $H_6$ is the balanced blow-up of an edge on $6$ vertices. 
Let $H_3^T$ denote an edge with an additional isolated vertex. Finally, Let $H_4^T$ denote the balanced blow-up of an edge on $4$ vertices, plus an additional isolated vertex, see Figure~\ref{fig:H6HT3HT4}.

\begin{figure}
\begin{center}
  \begin{tikzpicture}[flag_pic]\outercycle{6}{0}
\draw (x0) node[unlabeled_vertex]{};\draw (x1) node[unlabeled_vertex]{};\draw (x2) node[unlabeled_vertex]{};\draw (x3) node[unlabeled_vertex]{};\draw (x4) node[unlabeled_vertex]{};\draw (x5) node[unlabeled_vertex]{};
\labelvertex{0}\labelvertex{1}\labelvertex{2}\labelvertex{3}\labelvertex{4}\labelvertex{5}\drawhyperedge{0}{6}
\drawhypervertex{0}{0}
\drawhypervertex{2}{0}
\drawhypervertex{4}{0}
\drawhyperedge{1}{6}
\drawhypervertex{0}{1}
\drawhypervertex{2}{1}
\drawhypervertex{5}{1}
\drawhyperedge{2}{6}
\drawhypervertex{0}{2}
\drawhypervertex{3}{2}
\drawhypervertex{4}{2}
\drawhyperedge{3}{6}
\drawhypervertex{0}{3}
\drawhypervertex{3}{3}
\drawhypervertex{5}{3}
\drawhyperedge{4}{6}
\drawhypervertex{1}{4}
\drawhypervertex{2}{4}
\drawhypervertex{4}{4}
\drawhyperedge{5}{6}
\drawhypervertex{1}{5}
\drawhypervertex{2}{5}
\drawhypervertex{5}{5}
\drawhyperedge{6}{6}
\drawhypervertex{1}{6}
\drawhypervertex{3}{6}
\drawhypervertex{4}{6}
\drawhyperedge{7}{6}
\drawhypervertex{1}{7}
\drawhypervertex{3}{7}
\drawhypervertex{5}{7}
\draw (1.5,-2.5) node{$H_6$};
\end{tikzpicture} 
\hskip 4em
\begin{tikzpicture}[flag_pic]\outercycle{4}{0}
\draw (x0) node[unlabeled_vertex]{};\draw (x1) node[unlabeled_vertex]{};\draw (x2) node[unlabeled_vertex]{};\draw (x3) node[unlabeled_vertex]{};
\labelvertex{0}\labelvertex{1}\labelvertex{2}\labelvertex{3}\drawhyperedge{0}{4}
\drawhypervertex{0}{0}
\drawhypervertex{1}{0}
\drawhypervertex{2}{0}
\draw (0.75,-1.5) node{$H_3^T$};
\end{tikzpicture} 
\hskip 4em
\begin{tikzpicture}[flag_pic]\outercycle{5}{0}
\draw (x0) node[unlabeled_vertex]{};\draw (x1) node[unlabeled_vertex]{};\draw (x2) node[unlabeled_vertex]{};\draw (x3) node[unlabeled_vertex]{};\draw (x4) node[unlabeled_vertex]{};
\labelvertex{0}\labelvertex{1}\labelvertex{2}\labelvertex{3}\labelvertex{4}\drawhyperedge{0}{5}
\drawhypervertex{0}{0}
\drawhypervertex{2}{0}
\drawhypervertex{3}{0}
\drawhyperedge{1}{5}
\drawhypervertex{1}{1}
\drawhypervertex{2}{1}
\drawhypervertex{3}{1}
\draw (1,-1.5) node{$H_4^T$};
\end{tikzpicture}
\end{center}
\caption{Vertex-edge incidences of the hypergraphs $H_6$, $H_3^T$, and $H_4^T$.}\label{fig:H6HT3HT4}
\end{figure}

\begin{lemma}[\flags]\label{flag}
For both $G_n$ and $\bar{G}_n$, we have
    \begin{equation}\label{flagEq} 6 \cdot \frac{1}{15} H_6 + 0.196 \cdot \frac{1}{4} H_3^T + 0.366 \cdot \frac{1}{10} H_4^T \ge 
    0.0552798 +o(1).\end{equation}
\end{lemma}

\begin{proof}
We express in flag algebra that $G_n$ and $\bar{G}_n$ are $\{K_4^-,C_5^-\}$-free, with edge density at least $\frac14$, and almost regular in the sense of Claim~\ref{deg}. Now we run a plain flag algebra computation on 8 vertices (this is the single very large computation used in our proof) to obtain 
\begin{align*}
\frac{6}{15} H_6 + \frac{0.196}{4} H_3^T + \frac{0.366}{10} H_4^T
&\ge
\frac{2321754763443841429024560404249395956922591}{42000000000000000000000000000000000000000000}\\
&> 0.05527987532.
\end{align*}



\end{proof}

\begin{lemma}\label{lemMain}
    There exists some partition of the vertices of $G_n$ into $X_1,X_2,X_3,T$ such that 
    \begin{equation}\label{eqMain} 6x_1x_2x_3 - f_2 + 0.196t+0.366 \cdot t(1-t) \ge 0.221119+o(1).  
    \end{equation}
\end{lemma}

\begin{proof}
    For a slightly improved bound, we use $\bar{G}_n$ to partition the vertices. Note that $F_2$ in $G_n$ is contained in the corresponding set in $\bar{G}_n$.
    For any graph $G$ and any edge $e \in E(G_n)$, let $N_e(G)$ denote the number of subgraphs in $\bar{G}_n$ isomorphic to $G$ containing $e$. Then,
    \begin{align*}
    \begin{split}
        \sum_e 6N_e(H_6) + 0.196 n^2 &N_e(H_3^T) + 0.366 n N_e(H_4^T)  \\
        &= 48 {n \choose 6} H_6+ 0.196 n^2 {n \choose 4} H_3^T + 2 \cdot 0.366n {n \choose 5} H_4^T.
    \end{split}
     \end{align*} 

    Thus, by an averaging argument and by Lemma 
    \ref{flag}, there exists some edge $e$, such that 
    \begin{align}\label{eq1} \begin{split} \frac{1}{n^3}(6N_e(H_6) &+ 0.196 n^2 N_e(H_3^T) + 0.366 \cdot n N_e(H_4^T)) \\ 
    & \ge \frac{ 48 {n \choose 6} H_6+ 0.196 n^2 {n \choose 4} H_3^T + 2 \cdot 0.366 n {n \choose 5} H_4^T}{0.25 n^3{n \choose 3}} \\
    &= \frac{6 \cdot \frac{1}{15} H_6 + 0.196 \cdot \frac{1}{4} H_3^T + 0.366 \cdot \frac{1}{10} H_4^T}{0.25} + o(1) \\
    &\ge \frac{0.0552798}{0.25} + o(1) >0.221119+o(1).
    \end{split} \end{align}

    Now fix such an edge $e$, say consisting of vertices $v_1,v_2,v_3$. Then we may create a partition as follows: 
    If for a vertex $u$, the set $\{u\}\cup(\{v_1,v_2,v_3\}\setminus \{v_i\})$ induces an edge,  place $u$ in $X_i$. 
    Otherwise, place $u$ in $T$. Notice that $u$ appears only in one of the four sets since $G_n$ (and thus $\bar{G}_n$) is $K_4^-$-free. 
    Effectively what we are doing is placing $u$ in $X_i$ if it looks like $v_i$ to the other vertices in $e$. Now, this partition determines upper bounds for all $N_e(G)$. In particular, one can see that 
    \begin{align}\label{eq2} \begin{split}
        N_e(H_6) &= (X_1 - 1)(X_2 - 1)(X_3 - 1) - f_2{n\choose 3} \\
        N_e(H_3^T) &\le tn \\
        N_e(H_4^T) &\le t(1-t)n^2
    \end{split}
    \end{align}

    The result then follows by combining \eqref{eq1} and \eqref{eq2}. 
    
\end{proof}



Now that we see how Lemma \ref{flag} translates into a polynomial based on part sizes, we want to give a brief explanation why we chose the coefficients as we did. As we want to prove that $G_n$ is very close to $H_n$, we are guided by the structure of $H_n$.
In $H_n$, most edges span the three sets of the top level blow-up. Choosing such an edge, it is in close to $\frac{n^3}{27}$ copies of $H_6$. If we can somehow guarantee that there exists an edge in $G_n$ which is in that many copies of $H_6$, we can recover these three sets via our definition of the $X_i$, with $x_1=x_2=x_3=\frac13$. But since we can not pin point such an edge, and our argument instead relies on an average over all edges, other edges in $H_n$ reduce this average. To improve our bounds, we consider subgraphs of $H_n$ which are frequent supergraphs of other edges in $H_n$.

The second most frequent type of edges in $H_n$ lies completely in one set of the top-level, but spans the three sets of the next level. Using such an edge as our root yields
$x_1 = x_2 = x_3 = 1/9$ with $t=\frac23$. These edges are contained in many graphs isomorphic to $H_3^T$ and $H_4^T$, while the first type of edges is contained in neither. That is why we choose to include these graphs to boost our bound.

Our goal is to construct an equation from these three graphs so that only the top level edges of $H_n$ can beat the value we get from flag algebra, guaranteeing that an average edge is at the top level, and allows us to partition the vertices well into three parts. At the same time, we want that bound to be as high as possible, to give us the best possible partition.
In other words, we want \eqref{eqMain} to be feasible only close to $x_1 = x_2 = x_3 = \frac13$. However, the larger our polynomial is at $x_1 = x_2 = x_3 = \frac19$, the more this increases the right hand side of \eqref{flagEq}, and with this improves our bounds. Thus, we pick our coefficients for the $p(G)$ so that \eqref{eqMain} is not feasible at $x_1 = x_2 = x_3 = 1/9$, but is very close to being feasible to get the maximum contribution from those second level edges. In practice, this does take a bit of guess work. You start with the densities you have in $H_n$, and choose your coefficients with a bit of room to account for the bounds from flag algebra not being perfect. You may have to run the plain flag algebra method multiple times to test different coefficients in order to find the best possible combination. In our case, our first try with a very large computation was close enough to being successful so that it was sufficient to run a second much smaller computation to slightly adjust the coefficients. One could further optimize the coefficients or add an additional $p(G)$ term to get more control over the resulting polynomial, or to take into account edges on the next level of the construction. But we will see that for our needs, the bounds we get work sufficiently well. 


\begin{claim}\label{bounds}
For large enough $n$, for all partitions satisfying Lemma~\ref{lemMain},
\begin{align}
    x_3 &\ge 0.306 \nonumber\\
    x_1 &\le 0.361 \nonumber\\
    t &\le 0.0109 \nonumber\\
    f_2 &\le 0.001104 \nonumber\\
    f_2 &\le \tfrac29(1-t)^3 + 0.196t + 0.366 (1-t)t-
    0.221119\label{f2poly}
    \end{align} 
\end{claim}

\begin{proof}
 Note that for any fixed $t$, the left hand side of \eqref{eqMain} is maximized if $x_1 = x_2 = x_3=\frac{1-t}3$. Thus, $f_2$ is bounded 
 by~\eqref{f2poly}.
    This polynomial attains its maximum at $t=0$, implying that
    $f_2 \le 0.001104$ and $|F_2|<0.000184n^3$.
Further, the polynomial is negative for $0.0109<t\le 1$, implying the bound on $t$.

A little bit of calculus shows that decreasing $t$ to increase $x_1$ increases the left side
of \eqref{eqMain}, so it suffices to set $f_2 = t=0$ to bound $x_1$ and $x_3$. To bound $x_3$, we may further assume $x_1=x_2$, leading to a one-variable expression. The same is true for $x_1$ using $x_2=x_3$. Evaluating these expressions gives the claimed bounds.
\end{proof}


The previous claim shows that $G_n$ must be fairly close to the conjectured construction at the top level. Next, we want to establish improved bounds using Claim~\ref{deg}. In $\bar{G}_n$, every vertex has degree $\frac18 n^2+o(n^2)$.

Let us first look at the normalized degree sum inside $X_1$ (in $\bar{G}_n$). For ease of notation, we neglect "little o" terms in most of the remainder. We have
    \begin{align}\label{degsumx1}
        {\Sigma}_1 :=\frac1n\sum_{x\in X_1} \frac{d(x)}{n^2}=\frac18 x_1.
    \end{align}
We will further partition $\Sigma_1=\Sigma_I+\Sigma_R+\Sigma_F+\Sigma_T$, where $\Sigma_I$ is the contribution from edges completely in $X_1$, $\Sigma_R$ is the contribution from edges spanning all of $X_1$, $X_2$ and $X_3$, $\Sigma_F$ is the contribution from edges in $F_2$, and $\Sigma_T$ is the contribution from edges with at least one vertex in $T$.   
    
    We have $\Sigma_I\le\frac18 x_1^3$ --- note that the random process from $G_n$ to $\bar{G_n}$ also affected the degrees inside $X_1$. We have exactly $\Sigma_R=x_1x_2x_3-\frac16f_2$.

    Next, let us look at all edges of the types $u_1v_1v_2$, $u_1v_1v_3$, $v_1u_2v_2$, $v_1u_3v_3$ to bound $\Sigma_F$. The first two types each contribute $2$ to $\Sigma_F$, the last two each contribute $1$. To account for these edges, we will look at all subgraphs spanned by $6$ vertices $\{u_1,v_1,u_2,v_2,u_3,v_3\}$. Note that if $u_1v_1v_2\in E$, then at least one of $u_1v_2v_3$ and $v_1v_2v_3$ is in $F_2$ to avoid a $K_4^-$. 
\begin{claim}[\flags]
    For all $\{K_4^-,C_5^-\}$-free $6$-vertex $3$-graphs on vertices $\{u_1,v_1,u_2,v_2,u_3,v_3\}$ the ratio of the count of edges $\sum_{t \in \{u,v\}}2u_1v_1t_2 + 2u_1v_1t_3+ t_1v_2u_2 + t_1v_3u_3$ 
    over the number of non-edges intersecting all three pairs $u_1v_1$, $u_2v_2$, $u_3v_3$.
\end{claim}
\begin{proof}
    We generate all possible such $\{K_4^-,C_5^-\}$-free $6$-vertex $3$-graphs with a computer and find the ratio. Over all such $3$-graphs, the largest ratio is $\frac54$. Figure~\ref{fig:54} depicts the graph achieving the ratio $\frac54$.
\end{proof}

\tikzset{text_color0/.style={color=black}} 
\tikzset{text_color1/.style={color=red}} 
\tikzset{text_color2/.style={color=blue}} 
\tikzset{text_color3/.style={color=green!70!black}} 
\tikzset{text_color4/.style={color=orange}} 
\tikzset{text_color5/.style={color=gray}} 

\def\drawhypervertexBIcolor#1#2#3#4{ \pgfmathtruncatemacro{\plusone}{#1+1} \pgfmathtruncatemacro{\plusone}{#1+1} 
\draw[edge_color#3] (x#1)++(0,-0.2-0.2*#2)+(-0.2,0) -- +(0.2,0);
\draw[text_color#4] (x#1)++(0,-0.2-0.2*#2)  node[fill=white,outer sep=0,inner sep=0]{\tiny \plusone};
} 

\newcommand{\drawHECthree}[9]{
\drawhyperedge{#2}{#1}
\drawhypervertexBIcolor{#4}{#2}{#3}{#5}
\drawhypervertexBIcolor{#6}{#2}{#3}{#7}
\drawhypervertexBIcolor{#8}{#2}{#3}{#9}
}

\begin{figure}
\begin{center}
 \vc{  
  \begin{tikzpicture}[flag_pic]\outercycle{6}{0}
\draw (x0) node[vertex_color1]{};\draw (x1) node[vertex_color1]{};\draw (x2) node[vertex_color2]{};\draw (x3) node[vertex_color2]{};\draw (x4) node[vertex_color3]{};\draw (x5) node[vertex_color3]{};
\labelvertex{0}\labelvertex{1}\labelvertex{2}\labelvertex{3}\labelvertex{4}\labelvertex{5}
\drawHECthree{6}{0}{1}{0}{1}{1}{1}{4}{3}
\drawHECthree{6}{1}{1}{0}{1}{1}{1}{5}{3}
\drawHECthree{6}{2}{1}{0}{1}{2}{2}{3}{2}
\drawHECthree{6}{3}{2}{1}{1}{2}{2}{4}{3}
\drawHECthree{6}{4}{2}{1}{1}{2}{2}{5}{3}
\drawHECthree{6}{5}{2}{1}{1}{3}{2}{4}{3}
\drawHECthree{6}{6}{2}{1}{1}{3}{2}{5}{3}

\drawHECthree{6}{7}{3}{0}{1}{2}{2}{4}{3}
\drawHECthree{6}{8}{3}{0}{1}{2}{2}{5}{3}
\drawHECthree{6}{9}{3}{0}{1}{3}{2}{4}{3}
\drawHECthree{6}{10}{3}{0}{1}{3}{2}{5}{3}
\end{tikzpicture} } 

\end{center}
\caption{3-graphs on 6 vertices having 4 normal edges (blue), 3 extra funky edges (red) where the first two are with coefficient two, and 4 funky non-edges (green).}\label{fig:54}
\end{figure}

Our goal is to get a bound 
on $\Sigma_F$ in terms of non-edges in $F_2$. For this, let us set up an auxiliary bipartite weighted multigraph $M$ between the vertices $v_1\in X_1$ and non-edges $w_1w_2w_3\in F_2$. For every set $S=\{u_1,v_1,u_2,v_2,u_3,v_3\}$, we count the $s_1$ edges in $F_1\cap S$ containing $v_1$, and the $s_2$ non-edges $w_1w_2w_3\in F_2\cap S$. Then, we add an edge of weight $\frac{s_1}{s_2}$ in $M$ between $v_1$ and $w_1w_2w_3$. Notice that this is a multigraph since different $S$ contain $\{v_1,w_1,w_2,w_3\}$.

For every set $S$ containing $v_1$, the total edge weight of $v_1$ in $M$ is $s_1$.
So, counting over all sets $S$ containing $v_1$, the total weight of the edges in $M$ is at least $\frac12x_2x_3^2n^3\Sigma_F$, 
as each $F_1$-edge is in at least $\frac12x_2x_3^2n^3$ different sets $S$. On the other hand, every non-edge $w_1w_2w_3\in F_2$ is in $x_1x_2x_3n^3$ sets $S$, so it has edge weight in $M$ at most $\frac54 x_1x_2x_3n^3$ by the discussion before about $6$-vertex graphs, so the total edge weight is bounded above by $\frac54 x_1x_2x_3n^3f_2{n\choose 3}$. After normalizing the degree sum as before,
    
    \[
    \Sigma_F\le\frac1{n^3}\frac52f_2{n\choose 3}\frac{x_1x_2x_3}{x_2x_3^2}=\frac{5x_1}{12x_3}f_2.
    \]

    Finally, let us account for edges with vertices in $T$. 
Every vertex in $T$ has degree $\frac18n^2$, so there are at most $\frac18 tn^3$ edges containing a vertex in $T$ and a vertex in $X_1$. Edges with one vertex in $T$ and two vertices in $X_1$ count twice 
towards $\Sigma_T$.
    Since the link of every vertex is triangle free, there are at most $\frac14 tx_1^2n^3$ such edges.

    So, in total, 
    $\Sigma_T\le \frac18t(1+2x_1^2)$. 
%
    Taking everything together, we get
    \begin{align}\label{sideeq}
        \frac18 x_1=\Sigma_1\le \frac18 x_1^3+x_1x_2x_3-\frac16f_2+\frac{5x_1}{12x_3}f_2+\frac18t(1+2x_1^2),
    \end{align}
and consequently after multiplying by $24x_3$,
\begin{align}\label{sideeq2}
         3x_1x_3\le  3x_1^3x_3+24x_1x_2x_3^2-4x_3f_2+10x_1f_2+3x_3t+6x_1^2x_3 t.
    \end{align}
This inequality is weakest for maximized $f_2$, so we may assume that~\eqref{eqMain} holds with equality, and we have 
\begin{align}\label{sideeq3}
    0\le& 
-2.21119 x_1 + 1.96  x_1 t + 0.884476 x_3 + 2.216  x_3 t - 3 x_1 x_3 + 
 6  x_1^2 x_3t + 3 x_1^3 x_3\nonumber\\
 &+ 60 x_1^2 x_2 x_3 + 
 3.66 x_1  t - 3.66 x_1 t^2 - 
 1.464 x_3  t + 1.464 x_3 t^2 .
    %
\end{align}


 Combining~\eqref{eqMain} and \eqref{sideeq3}, we find new bounds on the $x_i$ 
 via optimization.
 This time, the optimization is less straight forward. This is a polynomial optimization problem which could certainly be treated with dedicated methods we are not experts in. Instead, we formulate a flag algebra model and solve that. We encode the problem as an edgeless graph with four vertex colors, where the size of the parts are $x_1,x_2,x_3,t$. 
 We can express all of our inequalities in this model, where monomials of degree $d$ correspond to graphs on $d$ vertices with suitable weights. We ask flag algebra to optimize the parameters however we like, giving us the following certified bounds with a moderately sized computation.
 \begin{claim}[\flags]\label{newbounds}
      \begin{align*}
     x_3 &\ge  0.31723\nonumber\\ 
     x_2x_3 &\ge 0.10613 \nonumber\\   
     x_1 &\le 0.33865 \nonumber\\
     x_1x_2&\le  0.11378\nonumber
     \end{align*} 
 \end{claim}
%
%
%

Next, we will bound a term which will show up later in our argument.
\begin{claim}[\flags]\label{f1f2}
    In $\bar{G}_n$, $f_1-f_2+\frac34(t^2-t^3)\le
0.0003042$.


\end{claim}
\begin{proof}
We are setting up a flag algebra model on $3$-graphs, in which vertices are partitioned into the $4$ parts. This additional partition restricts further the size of subgraphs we can compute with, but this smaller computation gives bounds sufficient for our purpose. Again, the $3$-graph has edge density $0.25$ and is $\{C_5^-,K_4^-\}$-free, and all bounds in Claims~\ref{bounds} and \ref{newbounds} apply to this partition. We now find the claimed upper bound for $f_1-f_2+\frac34(t^2-t^3)$ with the plain flag algebra method.
\end{proof}

From now on, we assume the partition $X_1\cup X_2\cup X_3\cup T$ is chosen so that the left side of \eqref{eqMain} is maximized. All bounds we have established hold for this partition.

Next, we will look at the link $L(v)$ of a single vertex $v$, the $2$-graph formed by the remaining pairs of vertices in edges containing $v$. The number of edges in $L(v)$ is $d(v)=\frac14{n\choose 2}+o(n^2)$. Further, this graph is triangle free as otherwise $v$ would be the degree $3$  vertex of a $K_4^-$ in $G_n$. 

Let $\{A,B,C\}=\{X_1,X_2,X_3\}$ such that the number of edges in $L(v)$ between $A$ and $B$ is maximized. This implies that $v\in C\cup T$.
We will use normalized sizes $a = |A|/n$, $b = |B|/n$ and $c = |C|/n$.

We want to find an upper bound for the number of edges in $L(v)$ inside $C$, $d_C(v)=d_C{n\choose 2}$. We know that the edge density in $C$ (in $\bar{G}\cup v$) is at most $e(C)\le \frac14$. 
This is still true if we blow up $v$ into a copy of $\bar{G}_{\mu n}$ for any $\mu\ge 0$, as the random sparcification has, with high probability, affected $d_C(v)$ proportionally. Therefore, for all $\mu\ge 0$,
\[
e(C){cn\choose 3}+\mu n d_C(v)+\frac14{\mu n\choose 3}\le
\frac14{(c+\mu)n\choose 3},
\]
so after dividing by ${n\choose 3}$,
\[
3\mu d_C\le \frac14(c^3+3c^2\mu+3c\mu^2)-e(C)c^3.
\]

On the other hand, every vertex in $T$ is in $\frac14{n\choose 2}+o(n^2)$ edges, so the total number of edges containing at least one vertex of $T$ is at most $\frac34 t{n\choose 3}$. 

For any $x,y,z> 0$ with $x+y+z=1$ we have
$\frac14(x^3+y^3+z^3)+6xyz\le \frac14$, with the maximum attained only for $x=y=z=\tfrac13$. As $a+b+c=1-t$, this scales to
$\frac14(a^3+b^3+c^3)+6abc\le \frac14(1-t)^3$.
Now, using also that $e(A),e(B)\le \frac14$,
\begin{align*}
    \tfrac14&\le e(C)c^3+e(A)a^3+e(B)b^3+6abc+f_1-f_2+\tfrac34 t\\
    &\le \tfrac14(a^3+b^3+c^3)+6abc-\tfrac14c^3+e(C)c^3+f_1-f_2+\tfrac34 t\\
    &\le \tfrac14(1-t)^3-\tfrac14c^3+e(C)c^3+f_1-f_2+\tfrac34 t.
\end{align*}
Therefore,
\[
e(C)c^3\ge \frac14 c^3-\frac34 t^2+\frac34 t^3-f_1+f_2.
\]
Together, this gives
\[
3\mu d_C\le \frac34(c^2\mu+c\mu^2)+\frac34(t^2-t^3)+f_1-f_2,
\]
and
\[
d_C\le \frac14(c^2+c\mu)+\frac1{3\mu}(\tfrac34(t^2-t^3)+f_1-f_2).
\]
This bound is minimized if
\[
\mu=2\sqrt{\frac{\tfrac34(t^2-t^3)+f_1-f_2}{3c}},
\]
so 
\[
d_C-\frac14c^2\le\sqrt{\frac{c(\tfrac34(t^2-t^3)+f_1-f_2)}{3}},
\]
and finally, for $d_C\ge \frac14c^2$,
\begin{align}
    3(d_C-\tfrac14c^2)^2\le c(\tfrac34(t^2-t^3)+f_1-f_2).
\end{align}

We need to establish two more bounds. For easier readability we will denote the set a vertex belongs to in its index.

Assume that $u_Av_A$ and $v_Av_B$ are edges in $L(v)$. Then, for any vertex $v_C$, $u_Avv_Av_Bv_Cu_A$ is a $C_5^-$ unless at least one of $u_Av_Bv_C$ and $v_Av_Bv_C$ is missing, and thus in $F_2$. We can count such sets of $4$ vertices containing a $P_3$ in $L(v)$ and a non-edge in $F_2$ in two ways. We can start with the $P_3$ and add any vertex in $C$ like we did above, giving a count of $p_3^{AAB}{n\choose 3}cn$ (giving implicitely the definitions of the density $p_3^{AAB}$ of such paths). 
Or, we can start with a non-edge in $F_2$, and add a vertex in $A$, giving a count of $f_2{n\choose 3}an$. The second choice is an overcount, as not every vertex in $A$ must yield a $P_3$, and as there can be more than one non-edge between these $4$ vertices. 
We use $\x{A},\x{B},\x{C},\x{T}$ to denote the single vertex graphs with the vertex in the respective set when writing the following inequalities in flag algebra notation. 
We get 
\begin{align}
    \p{AAB}*\x{C}&<f_2*\x{A},\\
    \p{ABB}*\x{C}&<f_2*\x{B},
\end{align}
where the second line follows from a symmetric argument.

For the second bound, assume that $u_Av_C$ and $v_Av_B$ are edges in $L(v)$. Then again, $vu_Av_Cv_Bv_Av$ is a $C_5^-$ unless at least one of $u_Av_Bv_C$ and $v_Av_Bv_C$ is missing, and thus in $F_2$. Denoting by $e^{AB}{n\choose 2}$ the number of edges from $A$ to $B$ (and at the same time $\e{AB}$ is an edge from $A$ to $B$ in flag algebra notation), similar counts as above give us $e^{AB}{n\choose 2}e^{AC}{n\choose 2}$ and $f_2{n\choose 3}|A|$ for the set of $4$ vertices, which yields
\begin{align}
    3\e{AB}*\e{AC}&< 2\x{A}*f_2,\\
        3\e{AB}*\e{BC}&< 2\x{B}*f_2.
\end{align}

Now we put all these bounds together, and formulate the following model for $L(v)$. As $f_1$ and $f_2$ do not exist inside the model, we replace them by their bounds
$f_2<(6\x{A}*\x{B}*\x{C}+ 0.196\x{T} + 0.366(1- \x{T})\x{T} - 0.221119)$ and
$f_1-f_2+\tfrac34 ( \x{T}^2-\x{T}^3)< 0.0003064$.

\begin{enumerate}[(a)]
    \item Four sets of vertices: $A,B,C,T$ 
    \item $\x{A}+\x{B}+\x{C}+\x{T}=\x{}=1$
    \item $0.31748<\x{A},\x{B},\x{C}<0.33833$
    \item $0.10622< \x{A}*\x{B},\x{A}*\x{C},\x{B}*\x{C}<0.11366$
    \item $0\le \x{T}<0.0109$
    \item triangle-free
    \item $\e{AB}\ge \e{AC},~\e{AB}\ge\e{BC}$
    \item  
    $\exx{}{} \geq \frac14$
    \item $6\x{A}*\x{B}*\x{C}+ 0.196\x{T} + 0.366(1- \x{T})\x{T} - 0.221119 \ge 0$
    \item 
        $\p{AAB}*\x{C}\le \x{A}*(6\x{A}*\x{B}*\x{C}+ 0.196\x{T} + 0.366(1- \x{T})\x{T} - 0.221119)$\\ 
        $\p{ABB}*\x{C}\le \x{B}*(6\x{A}*\x{B}*\x{C}+ 0.196\x{T} + 0.366(1- \x{T})\x{T} - 0.221119)$
    \item $3\e{AB}*\e{AC}\le 2\x{A}*(6\x{A}*\x{B}*\x{C}+ 0.196\x{T} + 0.366(1- \x{T})\x{T} - 0.221119)\\
        3\e{AB}*\e{BC}\le 2\x{B}*(6\x{A}*\x{B}*\x{C}+ 0.196\x{T} + 0.366(1- \x{T})\x{T} - 0.221119)$
    \item 
        $3(\e{CC}-\frac14\x{C}^2)^2\le \x{C}* 0.0003064$
\end{enumerate}

\begin{claim}[\flags]\label{cl:Lv}
    In $L(v)$, $\e{AB}>0.1942$. 
\end{claim}
\begin{proof}
    In our model for $L(v)$, we find a lower bound for $\e{AB}$ with a moderate flag algebra computation.
\end{proof}
\begin{claim}
    $T=\emptyset$.
\end{claim}
\begin{proof}
Assume that $v\in T$.
    As a member of $T$, it contributes (normalized) $0.196+0.366(1-t)\le 0.562$ to the left side of \eqref{eqMain}. If we move it to $C$, it contributes $6\cdot\tfrac12 e^{AB}>0.58$. 
    As the partition maximizes the left side, this is a contradiction. As $v$ was chosen arbitrarily, the claim follows.
\end{proof}
At this point, we could recompute and improve a few of our bounds using the fact that $T=\emptyset$. It turns out that this is not needed, so we skip this step.
\begin{claim}[\flags]\label{greenAB}
    Every vertex is in at most 
    $0.01452{n\choose 2}$ 
    non-edges in $F_2$. 
\end{claim}
\begin{proof}
    For an arbitrary vertex $v$, we use the model for $L(v)$ with $T=\emptyset$. Remember that $v\in C\cup T$, so in fact $v\in C$ here. Every non-edge in $L(v)$ between $A$ and $B$ corresponds to a non-edge in $F_2$ containing $v$.  Using a moderate flag algebra computation, we find an upper bound for $2 \x{A}*\x{B}-\e{AB}$, the density of such non-edges in $L(v)$. 
\end{proof}

While some of our arguments relied on being in $\bar{G}_n$, note that all resulting bounds apply to $G_n$ as well. In particular the only difference is that $F_1$ can only be larger in $G_n$, and $F_2$ can only be smaller. The only stated bound that may not apply to $G_n$ is the bound in Claim~\ref{f1f2}. But this bound was only used as an intermediate step to find further bounds for the partition, which then apply to $G_n$ as well as the partition is the same. The remaining argument starting here is in $G_n$.

\begin{claim}\label{maxdegree}
    In the partition $X_1\cup X_2\cup X_3$ of $G_n$, $F_1=F_2=\emptyset$.
\end{claim}

\begin{proof}
For ease of notation, indices of vertices indicate the index of the set they belong to.
If $|F_2|>|F_1|$, we can delete all edges in $F_1$ and add all edges in $F_2$ to create a graph with more edges, contradicting the extremality of $G_n$. So we may assume that $f_1\ge f_2$ and $F_1\ne\emptyset$.

For every $v_k$, either $v_iv_jv_k$ or $v_iu_jv_k$ is a funky non-edge to avoid a $K_4^-$, so each funky edge $v_iu_jv_j$ intersects at least $x_kn\ge x_3n$ funky non-edges in two vertices. Thus, there are at least $f_1{n\choose 3}x_3n$ such pairs of intersecting edges  and non-edges. Counting the pairs the other way and averaging, we see that there exists a funky non-edge intersecting at least $x_3n>0.31824n$
funky edges in two vertices. 

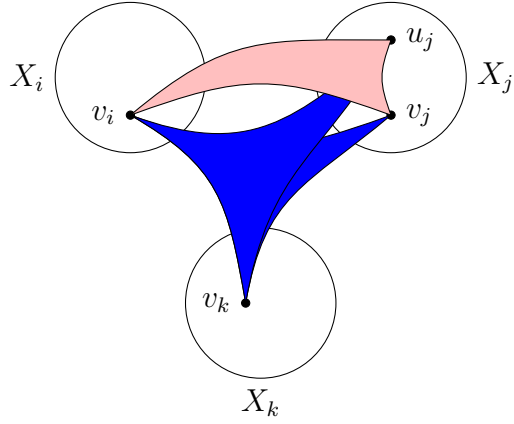
\begin{center}
\vc{
\begin{tikzpicture}
\draw
(30:2) coordinate(j) circle(1cm) ++(1,0)node[right]{$X_j$}
(150:2) coordinate(i) circle(1cm) ++(-1,0)node[left]{$X_i$}
(270:2) coordinate(k) circle(1cm) ++(0,-1)node[below]{$X_k$}

(i) ++(0,-0.5) coordinate(vi) node[vtx,label=left:$v_i$]{}
(j) ++(0,-0.5) coordinate(vj) node[vtx,label=right:$v_j$]{}
(j) ++(0,0.5) coordinate(uj) node[vtx,label=right:$u_j$]{}
(k) ++(-0.2,0) coordinate(vk) node[vtx,label=left:$v_k$]{}
;
\draw[hyperedgeB] 
(vi) to[out=-20,in=200,looseness=1.2] (vj) to[out=220,in=80,looseness=1.2] (vk) to[out=100,in=-30,looseness=1.2] (vi);

\draw[hyperedgeB] 
(vi) to[out=-20,in=220,looseness=1.2] (uj) to[out=240,in=80,looseness=1.2] (vk) to[out=100,in=-30,looseness=1.2] (vi);

\draw[hyperedgeP] 
(vi) to[out=20,in=160,looseness=1.2] (vj) to[out=110,in=250,looseness=1.2] (uj) to[out=180,in=40,looseness=1.2] (vi);
\end{tikzpicture}
}

\end{center}

Taking symmetry of $i,j,k$ into account, we may assume that $v_iv_jv_k$ is such a funky non-edge with more than 
$\frac{0.31824n}{6}= 0.05304n$ 
vertices $u_j$ such that $v_iu_jv_j$ is a funky edge. 
Let $Z_j=\{u_j:v_iu_jv_j\in E\}$, $z_jn=|Z_j|$. 

For any $w_i,w_k$, and $u_j\in Z_j$, notice that $v_jv_iu_jw_kw_iv_j$ is a $C_5^-$  unless one of $v_iu_jw_k$, $w_iu_jw_k$, $w_iv_jw_k$ is a funky non-edge.

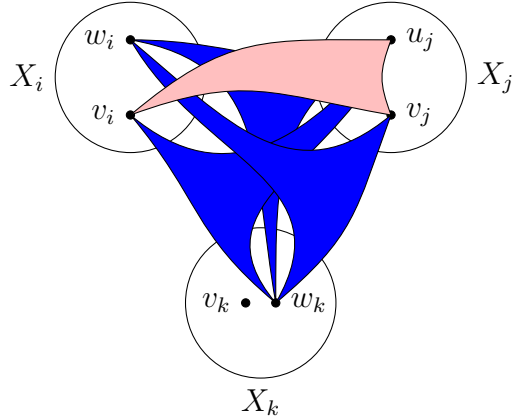
\begin{center}
\vc{
\begin{tikzpicture}
\draw
(30:2) coordinate(j) circle(1cm) ++(1,0)node[right]{$X_j$}
(150:2) coordinate(i) circle(1cm) ++(-1,0)node[left]{$X_i$}
(270:2) coordinate(k) circle(1cm) ++(0,-1)node[below]{$X_k$}

(i) ++(0,-0.5) coordinate(vi) node[vtx,label=left:$v_i$]{}
(i) ++(0, 0.5) coordinate(wi) node[vtx,label=left:$w_i$]{}
(j) ++(0,-0.5) coordinate(vj) node[vtx,label=right:$v_j$]{}
(j) ++(0,0.5) coordinate(uj) node[vtx,label=right:$u_j$]{}
(k) ++(-0.2,0) coordinate(vk) node[vtx,label=left:$v_k$]{}
(k) ++(0.2,0) coordinate(wk) node[vtx,label=right:$w_k$]{}
;

\draw[hyperedgeB] 
(wi) to[out=0,in=200,looseness=1.2] (uj) to[out=210,in=90,looseness=1.2] (wk) to[out=100,in=-10,looseness=1.2] (wi);

\draw[hyperedgeB] 
(vi) to[out=-40,in=230,looseness=1.2] (uj) to[out=250,in=130,looseness=1.2] (wk) to[out=140,in=-50,looseness=1.2] (vi);

\draw[hyperedgeB] 
(wi) to[out=-40,in=220,looseness=1.2] (vj) to[out=250,in=40,looseness=1.2] (wk) to[out=60,in=-60,looseness=1.2] (wi);

\draw[hyperedgeP] 
(vi) to[out=20,in=170,looseness=1.2] (vj) to[out=110,in=250,looseness=1.2] (uj) to[out=180,in=40,looseness=1.2] (vi);
\end{tikzpicture}
}
\end{center}
In total, by Claim~\ref{newbounds}, we have 
$x_iz_jx_kn^3\ge 0.10648z_jn^3$ 
choices for $w_i,u_j,w_k$. 

By Claim~\ref{greenAB}, we can not have many non-edges $w_iv_jw_k$, as they are all incident to $v_j$. To be precise, we can have fewer than
$0.0146{n\choose 2}<0.0073n^2$ 
such non-edges.
Thus, out of the more than $0.10648z_jn^3$ choices of $w_i,u_j,w_k$,
at least $(0.10648-0.0073)z_jn^3>0.0052n^3$ choices must yield a non-edge of the types $v_iu_jw_k$ or $w_iu_jw_k$.

 We can have at most $f_2{n\choose 3}<0.00018n^3$ non-edges $w_iu_jw_k$.  This still leaves more than $0.005n^3$ choices which force non-edges $v_iu_jw_k$. But since again by Claim~\ref{greenAB}, $v_i$ is in at most $0.0073n^2$ such non-edges, this accounts for at most $0.0073x_in^3<0.0025n^3$ choices, a contradiction proving the claim.
 %
\end{proof}

\begin{claim}\label{dense}
    The edge density $d$ in $G_n$ is $\frac14+o(1)$. Further, $x_i=\frac13+o(1)$ for all $i$. 
\end{claim}
\begin{proof}
    As $F_2=\emptyset$, the subgraphs spanned by the $X_i$ are extremal themselves, and thus have $(d+o(1)){|X_i|\choose 3}$ edges. Therefore,
    \[
    d=6x_1x_2x_3+(d+o(1))(x_1^3+x_2^3+x_3^3).
    \]
    For fixed $d<0.2502$ and $0.3<x_3\le x_2\le x_1<0.35$, the right side is maximized by $x_1=x_2=x_3=\frac13$. Solving for $d$ gives $d=\frac14+o(1)$.
\end{proof}
To finish the proof of Lemma~\ref{thmmain2}, the only part remaining to show is that the partition is balanced.
\begin{claim}\label{balance}
    For large enough $n$, $|X_1|-|X_3|\le 1$.
\end{claim}
\begin{proof}
    For the sake of contradiction, assume that $|X_1|-1\ge |X_3|+1$. Remove one vertex of minimum degree from $X_1$ and duplicate a vertex of maximum degree from $X_3$. This way we are deleting at most $(\tfrac14+o(1)){|X_1|\choose 2}$ edges from $X_1$, and add at least $(\tfrac14+o(1)){|X_3|\choose 2}$ edges to $X_3$, for a net loss of at most 
    \begin{align*}
        (\tfrac14+o(1))\left({|X_1|\choose 2}-{|X_3|\choose 2}\right)
        &=(\tfrac18+o(1))(|X_1|+|X_3|)(|X_1|-|X_3|)\\
        &=(\tfrac1{12}+o(1))(|X_1|-|X_3|)n
    \end{align*}
    edges inside these two sets. On the other hand, we are gaining 
    \begin{align*}
        (|X_1|-1)|X_2|(|X_3|+1)-|X_1||X_2||X_3|
        &=(|X_1|-|X_3|-1)|X_2|\\
        &=(\tfrac13+o(1)(|X_1|-|X_3|-1)n
    \end{align*}
    edges spanning all three sets. As $|X_1|-|X_3|\ge 2$,
    we have $2(|X_1|-|X_3|-1)\ge |X_1|-|X_3|$, so we gained a total of at least $(\tfrac1{12}+o(1))n$ edges, contradicting the maximality of $G_n$.
\end{proof}
This establishes Lemma~\ref{thmmain2}. 
\end{proof}

\section{Hypergraph limits and Tur\'an density}\label{limit}

The Hypergraph Removal Lemma can be deduced from the celebrated Hypergraph Regularity Lemma (independently proved by  R\"odl, Nagle,  Skokan, Schacht and Kohayakawa \cite{Rodl2005}, and Gowers \cite{Gowers2007}). While merely stating the regularity lemma is beyond the scope of this paper, the removal lemma is easy to state and use.
\begin{theorem}[Hypergraph Removal Lemma (\cite{Gowers2007}, \cite{Rodl2005})]
    Let $r\ge 2$, and let $H$ be an $r$-uniform hypergraph on $k$ vertices. Let $G$ be an $r$-uniform hypergraph on $n$ vertices containing $o(n^k)$ (not necessarily induced) copies of $H$. Then one may remove $o(n^r)$ edges from $G$ so that the resulting hypergraph is $H$-free.
\end{theorem}

With this, Theorem~\ref{Turan} easily follows from Lemma~\ref{thmmain2}.
\begin{proof}[Proof of Theorem~\ref{Turan}]
    By Proposition~\ref{noK4-}, we can use the Hypergraph Removal Lemma to remove $o(n^3)$ edges from $F_n$ to arrive at a $\{C_5^-,K_4^-\}$-free graph $F_n'$. Writing $e(H)$ for the edge density of a hypergraph $H$, we have
\[
e(G_n)+o(1)=e(H_n)\le e(F_n)=e(F_n')+o(1)\le e(G_n)+o(1)=\frac14+o(1),
\]
implying $e(F_n)=\frac14+o(1)$ and thus
Theorem~\ref{Turan}.
\end{proof}

For a $3$-graph $H$, we write $H[t]=H(E_t,E_t,\ldots,E_t)$ for the balanced blow-up of $H$ by $|H|$ copies of the empty $3$-graph on $t$ vertices.
Proposition~\ref{noK4-} is a special instance of a more general Theorem (see section 2 in~\cite{Keevash2011}).
\begin{theorem}\label{blow}
    For every $3$-graph $H$ and positive integer $t$, if a $3$-graph $G$ on $n$ vertices has no subgraph isomorphic to $H[t]$, then $p(H,G)=o(1)$.
\end{theorem}
For a vertex $v$, let us denote copies of $v$ by $v^1,v^2,\ldots$. Proposition~\ref{noK4-} then follows from Theorem~\ref{blow} since $C_5^-$ is in $K_4^-[2]$ as can be seen by the sequence
$13243^11$.
In~\cite{balogh2023turan}, 
    it is shown that 
    \begin{proposition}\label{clTurgen}
        Every $C_\ell^-$ with $\ell\ge 7$  and $\ell$ not divisible by $3$ is contained in a blow-up of $C_5^-$.
    \end{proposition}
    \begin{proof}
        We paraphrase the argument, and first note that $13243^154^11$ gives a $C_7^-$ in $C_5^-[2]$.
    Next, observe that $C_{k+3}^-$ is contained in $C_k^-[2]$ by considering the string $1231^12^13^14\ldots k1$. The Claim now follows by induction starting in $C_5^-$ and $C_7^-$.
    \end{proof}

\begin{proof}[Proof of Theorem~\ref{Turangen}]
    To prove the theorem, note that Theorem~\ref{blow} and Proposition~\ref{clTurgen} imply that for every $\ell\ge 7$ not divisible by $3$, every $C_\ell^-$-free graph has $C_5^-=o(1)$. By the Hypergraph Removal Lemma, we can delete $o(n^3)$ edges to destroy all copies of $C_5^-$, and then Theorem~\ref{Turan} implies that $ex\left(C_\ell^-,n\right)\le (\tfrac14+o(1){n\choose 3}$. As $H_n$ does not contain any copies of $C_\ell^-$, the theorem follows.
\end{proof}

``Graphons, short for graph functions, are the
limiting objects for sequences of large, finite
graphs with respect to the so-called cut metric.
They were introduced and developed by Borgs,
Chayes, Lov\'asz, S\'os, Szegedy,
and Vesztergombi in \cite{Borgs} and \cite{LovSze}.'' A short explainer by Glasscock, from which we cited the previous sentence, can be found in \cite{Glass}, and an extensive treatment by Lov\'asz appeared in \cite{LovaszGL}.

Later in \cite{ElekSze}, Elek and Szegedy 
developed the corresponding theory of hypergraphons, which are limit objects of hypergraphs. Similarly and closely related to the Hypergraph Regularity Lemma, stating the definition of a hypergraphon is beyond the scope of this paper, an introduction is contained in \cite{LovaszGL}. For our purposes, it is enough to consider the following theorem.
\begin{theorem}[\cite{ElekSze}]
    Let $(P_n)$ and $(Q_n)$ be two sequences of $r$-uniform hypergraphs on $n$ vertices. Then $(P_n)$ and $(Q_n)$ converge to the same unique hypergraphon, if and only if for every $r$-uniform hypergraph $H$, $\lim_{n\to\infty} p(H,P_n)=\lim_{n\to\infty} p(H,Q_n)$.
\end{theorem}
It follows easily from an inductive application of Theorem~\ref{thmmain2} that $(H_n)$ and $(G_n)$ converge to the same hypergraphon $W$. In fact, we can label their vertices in a way that their edit distance is $o(n^3)$, a property stronger than their convergence to the same limit.

We say that a hypergraphon $R$ is finitely forcible if there exists a finite set $\{P_1,P_2,\ldots,P_k\}$ of $r$-uniform hypergraphs and real numbers $p_1,\ldots,p_k\in [0,1]$ so that for all sequences $(Q_n)$ of $r$-uniform hypergraphs on $n$ vertices, $\forall i:\lim p(P_i,Q_n)=p_i$ implies that $R$ is the limit of $(Q_n)$. In our next statement, we show that $W$ is finitely forcible with only two subhypergraph densities.

\begin{theorem}\label{thmlimit}
    Let $(Q_n)$ be a sequence of $3$-uniform hypergraphs on $n$ vertices with\\ $\lim p(C_5^-,Q_n)=0$ and $\lim p(K_3,Q_n)=\frac14$. Then $(Q_n)$ converges to $W$, the limit of $(H_n)$.
\end{theorem}
\begin{proof}
    Using the Hypergraph Removal Lemma, Theorem~\ref{blow}, Proposition~\ref{clTurgen}, and the Hypergraph Removal Lemma a second time, we can delete $o(n^3)$ edges from $Q_n$ to construct $\{C_5^-,K_4^-\}$-free hypergraphs $Q_n'$ with $p(H,Q_n)=p(H,Q_n')+o(1)$ for every $3$-uniform $H$. Theorem~\ref{Turan} implies that $\Delta (Q_n')=(\tfrac14+o(1)){n\choose 2}$. Let $\eps=\eps(n)>0$ with $\lim_{n\to\infty} \eps(n)=0$ to be chosen later, and define
    \[
    S_n=\left\{v\in V(Q_n):d(v)<(e(Q_n')-\eps){n\choose 2}\right\}
    \]
    be the set of small degree vertices. If $\rho {n\choose 2}=\Delta(Q_n')-e(Q_n'){n\choose 2}$, then $|S_n|\eps\le n\rho<n\rho+1$, so choosing $\eps=\sqrt{\rho+\tfrac1n}=o(1)$ gives $|S_n|<\eps n=o(n)$. Let $Q_n''=Q_n'-S_n$, then $\Delta(Q_n'')-\delta(Q_n'')=o(n^2)$. 

    With this and a bit of care handling the $o(1)$ terms, we can apply all steps  of the proof of Theorem~\ref{thmmain2} until before Claim~\ref{maxdegree} to $Q_n''$ instead of $\bar{G}_n$. Claim~\ref{maxdegree} is the first time we use extremality of $G_n$ in the proof.
    
    With essentially the same proof as in Claim~\ref{maxdegree}, we can instead prove that $f_1=o(1)$ and $f_2=o(1)$. The same then still holds if we add $S_n$ to $X_1$. This suffices to show that $x_i=\frac13+o(1)$, and we have therefore, after adding back the $o(n^3)$ edges we removed in the very beginning, the following claim.
    \begin{claim}
        The vertices of $Q_n$ can be partitioned $V(Q_n)=X_1\cup X_2\cup X_3$ with $x_i=\frac13+o(1)$ such that $f_1=o(1)$ and $f_2=o(1)$, and the induced hypergraphs on the $X_i$ each have edge density $e(X_i)=\frac14+o(1)$. 
    \end{claim}
    Applying this claim inductively to the $X_i$ proves the theorem.
\end{proof}

\section{Proofs of Theorems~\ref{thmmain} and~\ref{exno} 
}\label{main}

\restatethmmain*

\begin{proof}[Proof of Theorem~\ref{thmmain}]
    By Theorem~\ref{thmlimit}, $(F_n)$ converges to $H_n$. It remains to show that this convergence is even stronger, and $F_1=F_2=\emptyset$ if we more carefully partition $S_n$ from the proof of Theorem~\ref{thmlimit} across the $X_i$. Once we prove this, Claim~\ref{balance} and thus the full theorem follows.

We define $F_n',F_n'',S_n$ as in the proof of Theorem~\ref{thmlimit} and proceed in the proof until Claim~\ref{maxdegree}. Add back the missing edges of $F_n$, and partition $S_n$ across the $X_i$ to maximize $6x_1x_2x_3-f_2$. This only affected $o(n^3)$ edges, so we still have $f_1=o(1)$ and $f_2=o(1)$. If $f_2>f_1$, we can increase the number of edges without creating a $C_\ell^-$ if we replace all edges in $F_1$ by all edges in $F_2$, a contradiction showing that $f_2\le f_1$.

In the following, we describe the argument for $\ell=5$, but it is straight forward to adapt it to general $\ell$.

Let us now consider the link of a vertex $L(v)$ again, for $\{A,B,C\}=\{X_1,X_2,X_3\}$. We get a new model with the new bounds we know for any vertex $v$, by symmetry we may assume $v\in C$. As flag algebra computes in the limit, we omit $o(1)$ terms.

\begin{enumerate}[(a)]
    \item Three sets of vertices: $A,B,C$ 
    \item \sout{$\x{A}+\x{B}+\x{C}=1$}
    \item $\x{A},\x{B},\x{C}=\frac13$
    \item \sout{$0.10649< \x{A}*\x{B},\x{A}*\x{C},\x{B}*\x{C}<0.11354$}\label{d}
    \item \sout{$0\le \x{T}<0.01022$}
    \item \sout{triangle-free}\label{f}
    \item $\e{AB}\ge \e{AC},\e{BC}$\label{g}
    \item $\e{AA}+\e{BB}+\e{CC}+\e{AB}+\e{AC}+\e{BC}= \exx{}{} =\frac14$\label{h}
    \item \sout{$6\x{A}*\x{B}*\x{C}+ 0.196\x{T} + 0.366(1- \x{T})\x{T} - 0.22118 \ge 0$}
    \item \label{j}
        $\p{AAB}+\kkk{AAB}= 0\\ 
        \p{ABB}+\kkk{ABB}= 0$
    \item \label{k}
    $\e{AC}=0\\
        \e{BC}=0$
    \item \label{l}
        $\e{CC}\le \frac14 \x{C}^2=\frac1{36}\\
        \e{AA},\e{BB}\le \frac1{36}$
\end{enumerate}
Let us give some explanations. We do not need (\ref{d}) anymore. 
The link of $v$ may contain triangles. Denoting the corresponding triangle densities by $k_3^{AAB}$ and $k_3^{ABB}$, (\ref{j}) follows from the same argument as before. Similarly, (\ref{k}) follows together with (\ref{g}). Note that neither of these arguments used that $G_n$ is $K_4^-$-free. 
Finally, (\ref{l}) follows for $e^{CC}$ with the same argument, but there are more straight forward ways to prove it. Symmetrical arguments apply to $e^{AA}$ and $e^{BB}$.

From (\ref{j}) we see that all but $o(n)$ vertices in $A$ have either $o(n)$ neighbors in $A$ or $o(n)$ neighbors in $B$. A symmetric statement is true for vertices in $B$. Denoting the number of vertices in $A$ with more than $o(n)$ neighbors in $A$ by $s_An$, and  the number of vertices in $B$ with more than $o(n)$ neighbors in $B$ by $s_Bn$, we have   $0\le s_A,s_B\le \frac13$, and
\[
e^{AA}+e^{AB}+e^{BB}\le s_A^2+2(\tfrac13-s_A)(\tfrac13-s_B)+s_B^2+o(1).
\]
This counts all the $e^{AA}$- and $e^{BB}$-edges in two cliques of sizes $s_An$ and $s_Bn$, and the $e^{AB}$-edges in a complete bipartite graph not overlapping the cliques. But then, (\ref{h}) and (\ref{l}) imply that $e^{AA}+e^{AB}+e^{BB}=\frac29+o(1)$, which is only possible for $s_A=s_B=o(1)$ (noting that $e^{AA}=\frac19+o(1)$ is ruled out by (\ref{j})), and $\e{AB}=2\x{A}*\x{B}+o(1)$.
We conclude that 
\begin{claim}\label{smallF2}
 The number of  non-edges in $F_2$ containing  $v$ is $o(n^2)$.
\end{claim}
The remainder of the proof is again similar to Lemma~\ref{thmmain2}, but easier. Assume for the sake of contradiction that $F_1\ne \emptyset $, and $u_Av_Av_B\in F_1$. Then for any choice of $w_B,w_C$, $u_Av_Bv_Aw_Cw_Bu_A$ is a $C_5^-$ unless at least one of $u_Aw_Bw_C,v_Aw_Bw_C,v_Av_Bw_C$ is missing. By Claim~\ref{smallF2}, $u_Aw_Bw_C$ and $v_Aw_Bw_C$ can miss only $o(n^2)$ times, so for the remaining $\frac{n^2}9+o(n^2)$ choices, $v_Av_Bw_C\in F_2$ (and in fact $u_Av_Bw_C\in F_2$ as well considering $v_Av_Bu_Aw_Cw_Bv_A$). In other words, $u_Av_Av_B$ intersects at least $\frac23 n+o(n)$ non-edges in $F_2$ in two vertices. Counting this structure for all edges in $F_1$, and noting that $f_1\ge f_2$, this implies that there exists some non-edge $v_Av_Bv_C\in F_2$ which intersects at least $\frac23 n+o(n)$ edges in $F_1$.

In particular, we may assume by symmetry that $v_Av_B$ is in at least $\frac19n+o(n)$ edges of type $v_Av_Bu_B$. For any pair $w_A,w_C$, and any $u_B$ with $v_Av_Bu_B\in F_1$, at least one of $w_Au_Bw_C$, $v_Au_Bw_C$, $w_Av_Bw_C$ must be missing. But there are $o(n^3)$ non-edges of type $w_Au_Bw_C$ in $F_2$. Further, by Claim~\ref{smallF2}, there are $o(n^2)$ non-edges of types $v_Au_Bw_c$ and $w_Av_Bw_c$, each accounting for at most $\frac13n+o(n)$ choices of the at least $\frac1{81}n^3+o(n^3)$ choices of $w_A,u_B,w_C$. This contradiction shows that $F_1=F_2=\emptyset$.

The only remaining Claim~\ref{balance} from Lemma~\ref{thmmain2} follows with the same proof, establishing Theorem~\ref{thmmain}.
\end{proof}

\restateexno*

\begin{proof}[Proof of Theorem~\ref{exno}]
    As described in Section~\ref{intro}, for some large enough $M$, for every $k$ and $3^{k+1}M>n\ge 3^kM$, $F_n$ agrees on the first $k$ levels of $H_n$ from the outside in. This shows that $F_n$ and $H_n$ are isomorphic up to changing $3^k{M\choose 3}=O(n)$ edges.

    Thus, it suffices to show that 
    \[
\left| ||H_n||-\frac{n^3}{24}\right| <\frac16n\log_3n+O(n).
    \]
    For this, we use induction on $n$ and consider three cases $n=3k$, $n=3k+1$, and $n=3k+2$. 
    Note that the start of the induction is trivial due to the $O(n)$ term.

    Then we have for $n=3k$,
    \begin{align*}
        \left| ||H_{3k}||-\frac{(3k)^3}{24}\right|&= \left| k^3+3||H_k||-\frac{(3k)^3}{24}\right|\\
        &\le  \left| k^3+3\frac{k^3}{24}-\frac{(3k)^3}{24}\right|+\frac36k\log_3k+O(3k)\\
        &< \frac16n\log_3n+O(n).
    \end{align*}
Similarly for $n=3k+1$,
\begin{align*}
        \left| ||H_{3k+1}||-\frac{(3k+1)^3}{24}\right|=& \left| k^3+k^2+2||H_k||+||H_{k+1}||-\frac{(3k+1)^3}{24}\right|\\
        \le&  \left| k^3+k^2+2\frac{k^3}{24}+\frac{(k+1)^3}{24}-\frac{(3k+1)^3}{24}\right|\\
        &+\frac16(3k+1)\log_3(k+1)+O(3k+1)\\
        <&~  \frac{6k}{24}+\frac16n(\log_3n-\tfrac12)+O(n)\\
        <&~ \frac16n\log_3n+O(n),
    \end{align*}
and for $n=3k+2$,
    \begin{align*}
        \left| ||H_{3k+2}||-\frac{(3k+2)^3}{24}\right|=& \left| k^3+2k^2+k+||H_k||+2||H_{k+1}||-\frac{(3k+2)^3}{24}\right|\\
        \le&  \left| k^3+2k^2+k+\frac{k^3}{24}+2\frac{(k+1)^3}{24}-\frac{(3k+2)^3}{24}\right|\\
        &+\frac16(3k+2)\log_3(k+1)+O(3k+2)\\
        <&~  \frac{6k+6}{24}+\frac16n(\log_3n-\tfrac23)+O(n)\\
        <&~ \frac16n\log_3n+O(n).
    \end{align*}
\end{proof}

\section*{Acknowledgments}
This work used the computing resources at the Center for Computational Mathematics, University of Colorado Denver, including the Alderaan cluster, supported by the National Science Foundation award OAC-2019089.

After publishing a preprint of this manuscript, we learned that Bodn\'{a}r, Le\'on, Liu, and Pikhurko are also working on this problem and are close to a final independent manuscript.

\bibliographystyle{abbrv}
\bibliography{references}

\end{document}